\documentclass[a4paper]{amsart}

\pdfoutput=1

\usepackage{amsmath,amssymb,amsthm,enumerate,graphicx,overpic}
\usepackage{hyperref}

\usepackage[margin=1.1in]{geometry}

\newtheorem{thm}{Theorem}[section]
\newtheorem{lemma}[thm]{Lemma}

\newtheorem{cor}[thm]{Corollary}

\newtheorem{thmintro}{Theorem}

\theoremstyle{definition}

\newtheorem{rmk}[thm]{Remark}
\newtheorem{ex}[thm]{Example}
  
\DeclareMathOperator{\Ends}{Ends}

\DeclareMathOperator{\Isom}{Isom}
\newcommand{\LL}{\mathcal{L}}
\newcommand{\Z}{\mathbb{Z}}
\newcommand{\R}{\mathbb{R}}
\newcommand{\N}{\mathbb{N}}
\DeclareMathOperator{\arcsinh}{arcsinh}
\newcommand{\st}{\;|\;}
\newcommand{\ssm}{\smallsetminus}

\DeclareMathOperator{\is}{Is}

\title{Isospectral hyperbolic surfaces of infinite genus}
\author{Federica Fanoni}
\address{CNRS, Univ Paris Est Creteil, Univ Gustave Eiffel, LAMA, F-94010 Creteil, France}
\email{federica.fanoni@math.cnrs.fr}
\date{\today}

\begin{document}

\begin{abstract}We show that any infinite-type surface without planar ends admits arbitrarily large families of length isospectral hyperbolic structures. If the surface has infinite genus and its space of ends is self-similar, we construct an uncountable family of isospectral and quasiconformally distinct hyperbolic structures.
\end{abstract}
\maketitle

\section{Introduction}
Given a hyperbolic surface, its length spectrum is the collection of lengths of primitive closed geodesics, counted with multiplicities. Surfaces with the same length spectrum are called isospectral.

The study of isospectrality for hyperbolic surfaces of finite type (i.e.\ with finitely generated fundamental group) has a long history. It was first shown by Huber \cite{huber_analytischen} and M\"uller \cite{mueller_spectral} that the length spectrum determines the topology of the surface and by McKean \cite{mckean_selberg} and M\"{u}ller \cite{mueller_spectral} that there can be only finitely many surfaces isospectral to a given one. Wolpert \cite{wolpert_length} proved that \emph{generically} the length spectrum determines the hyperbolic structure: if $\mathcal{V}_g$ denotes the subset of Teichm\"{u}ller space of a closed surface of genus $g\geq 2$ given by hyperbolic structures $X$ such that there exists a hyperbolic structure $Y$ isospectral (and non-isometric) to $X$, then $\mathcal{V}_g$ is a proper real analytic subvariety of Teichm\"uller space. In \cite{buser_isospectral}, Buser showed that $\dim \mathcal{V}_g>0$ if $g=5$ or $g\geq 7$.

Lots of effort has been put into building families of isospectral hyperbolic surfaces, and we refer the interested reader to Gordon's survey \cite[Chapter 6]{dv_handbook} for references to these results. We just mention here the construction of the largest known family: in \cite{bgg_mutually}, Brooks, Gornet and Gustafson proved that there are infinitely many genera $g$ for which there exist a family of pairwise isospectral closed hyperbolic surfaces of genus $g$ of size growing like $g^{c\log(g)}$ (where $c$ is a universal constant).

On the other hand, in the case of a closed surface of genus $g\geq 2$, Buser \cite[Chapter 13]{buser_geometry} gave an upper bound, depending only on $g$, on the cardinality of a family of isospectral hyperbolic structures on the surface. Buser's bound was improved by Parlier \cite{parlier_interrogating}, who proved that the cardinality of such a family does not exceed $g^{Cg}$ (for a universal constant $C$). See also Ungemach \cite{ungemach_bound} for related work.

When it comes to surfaces of infinite type, very little is known about the length spectrum. The first striking difference with the finite-type case is that the length spectrum of a hyperbolic surface of infinite type does not need to be discrete: it is easy to construct examples of infinite-type hyperbolic surfaces with infinitely many primitive closed geodesics of length bounded by a constant. Basmajian and Kim \cite{bk_geometrically} gave necessary and sufficient conditions for an infinite-type hyperbolic surface to have discrete spectrum and constructed infinite dimensional family of quasiconformally distinct hyperbolic structures having a discrete (respectively, non-discrete) length spectrum.

The goal of this paper is to show that infinite-type hyperbolic surfaces have large families of isospectral hyperbolic structures. Given a hyperbolic surface $X$, we denote by $\is(X)$ the cardinality of the set of hyperbolic surfaces isospectral to $X$ (where hyperbolic surfaces are considered up to isometry). The first result we prove is the following:

\begin{thmintro}\label{thm:largefamilies}
Let $S$ be an infinite-type surface without planar ends. Then for every $n\in N$ there is an infinite-dimensional family of hyperbolic structures $X$ on $S$ with $\is(X)\geq n$.
\end{thmintro}


With different requirements on the topology of the base surface we can construct much larger isospectral families. Moreover we show that, under these assumptions, not only these surfaces are not determined, up to isometry, by their length spectrum, but that they are not even determined up to quasiconformal equivalence.

\begin{thmintro}\label{thm:selfsimilar}
Let $S$ be an infinite-genus surface with self-similar endspace. Then there is an infinite-dimensional family of hyperbolic structures $X$ on $S$ such that:
\begin{enumerate}[(i)]
\item $\is(X)$ is uncountable, and
\item uncountably many surfaces in $\is(X)$ are pairwise not quasiconformal to each other.
\end{enumerate}
\end{thmintro}

We refer to Section \ref{sec:prelim-topology} for the definition of self-similarity. We just note that, up to homeomorphism, there are uncountably many infinite-genus surfaces with self-similar endspace (see \cite[Section 4.2]{apv_isometry}). Simple examples are the Loch Ness monster (the surface with a single end, that is nonplanar) and the blooming Cantor tree (the surface without planar ends, whose space of ends is a Cantor set).

We also remark that recently Baik, Choi and Kim \cite{bck_simple} studied the \emph{simple} length spectrum (the multiset of lengths of simple closed geodesics) of surfaces of finite- and infinite-type. They proved that \emph{generically} (i.e.\ away from a meagre subset of Teichm\"uller space) hyperbolic surfaces are determined by their simple length spectrum. Note that, in contrast with the case of the length spectrum, it is not known whether there are pairs of non-isometric hyperbolic surfaces with the same simple length spectrum.

\subsection{Tools for the proofs}
To prove isospectrality in Theorem \ref{thm:largefamilies} we will rely on Sunada's criterion \cite{sunada_riemannian} (see also \cite{buser_geometry}, for a statement not requiring compactness). Sunada's criterion holds for a large class of manifolds and for both the spectrum of the Laplacian and the length spectrum, but for our purposes we will only need the following (see Section \ref{sec:groups} for the definition of almost conjugate subgroups):

\begin{thm}[Sunada (\cite{sunada_riemannian})]\label{thm:sunada}
Let $M$ be a complete hyperbolic surface. Suppose $G$ is a finite group acting on $M$ by isometries with finitely many fixed points. If $H_1$ and $H_2$ are almost conjugate subgroups of $G$, acting on $M$ without fixed points, the quotient surfaces $M/H_1$ and $M/H_2$ are isospectral.
\end{thm}
To construct the covering surface we will rely on the construction of infinite-genus hyperbolic surfaces with a given (finite) isometry group due to Aougab, Patel and Vlamis \cite{apv_isometry}.

We will use the same construction also for the proof of Theorem \ref{thm:selfsimilar}. The main difference with Theorem \ref{thm:largefamilies} is that these surfaces admit hyperbolic structures with a \emph{countably infinite} isometry group, as shown by Aougab, Patel and Vlamis \cite{apv_isometry}. While this will allow us to construct infinite isospectral families, it will also mean that we won't be able to apply Sunada's result directly. Instead, we will show isospectrality  by applying the \emph{transplantation of geodesics} technique, first introduced by Buser in \cite{buser_isospectral} (see also \cite{berard_transplantation}).

\section*{Acknowledgements}
The author would like to thank Bram Petri and Nick Vlamis for useful conversations.

\section{Preliminaries}\label{sec:preliminaries}

Surfaces will be assumed to be orientable and connected and, unless otherwise stated, they will have no boundary. If they do, the boundary is a union of compact components. Surfaces are \emph{of finite type} if their fundamental group is finitely generated and \emph{of infinite type} otherwise.

\subsection{Topology}\label{sec:prelim-topology}
Surfaces without boundary are topologically classified by their genus and the pair of topological spaces $(\Ends(S),\Ends_g(S))$, where $\Ends(S)$ is the space of ends of the surface and $\Ends_g(S)$ is the (closed) subspace of nonplanar ends, as shown by Ker\'ekj\'art\'o \cite{kerekjarto_vorlesungen} and Richards \cite{richards_classification}. We refer to Aramayona and Vlamis' survey \cite{av_big} for definitions and properties of these objects.

We will say that an end is \emph{accumulated by boundary components} if any open neighborhood of the end in the Freudenthal compactification $S\cup \Ends(S)$ contains boundary components. We denote by $\Ends_\partial(S)$ the set of ends accumulated by boundary components. It is not difficult to deduce from the classification of surfaces without boundary that two surfaces $S$ and $S'$ with compact boundary components are homeomorphic if and only if they have the same genus and the same number of boundary components and there is a homeomorphism $f:\Ends(S)\to\Ends(S')$ sending $\Ends_g(S)$ to $\Ends_g(S')$ and $\Ends_\partial(S)$ to $\Ends_\partial(S')$.

We will be interested in the subclass of infinite-type surfaces whose space of ends is self-similar -- a condition first introduced by Mann and Rafi \cite{mr_large}. We say that a surface has \emph{self-similar} space of ends if for every decomposition
$$\Ends(S)=U_1\sqcup\dots\sqcup U_n$$
into (pairwise disjoint) clopen subsets $U_i$, there is $i\in\{1,\dots, n\}$ and an open $A\subset U_i$ such that
$$(A,A\cap \Ends_g(S))\simeq (\Ends(S),\Ends_g(S)).$$
Among surfaces without planar ends, any surface whose space of ends is either a Cantor set or of the form $\omega^\alpha+1$, for some countable ordinal $\alpha$, has self-similar endspace.

In \cite[Theorem 5.2]{apv_isometry}, Aougab, Patel and Vlamis showed that self-similarity is equivalent to radial symmetry, where the space of ends of a surface is said to have \emph{radial symmetry} if it is either a single point or there is $x\in\Ends(S)$ such that
$$\Ends(S)\ssm\{x\}\simeq \bigsqcup_{n\in N}E_n,$$
where the $E_n$ are (pairwise disjoint) noncompact subsets of $\Ends(S)$ such that for every $n,m\in N$
$$(E_n,E_n\cap\Ends_g(S))\simeq (E_m,E_m\cap\Ends_g(S)).$$
The point $x$ is called a \emph{star point} of $\Ends(S)$. Note that for every $n$, the closure $\overline{E}_n$ of $E_n$ is $E_n\cup\{x\}$.

\subsection{Hyperbolic geometry}\label{sec:prelim-geometry}
A \emph{hyperbolic surface} $X$ is a surface endowed with a hyperbolic metric (a Riemannian metric of constant curvature $-1$), which we require to be complete and \emph{of the first kind} (i.e.\ equal to its convex core). Recall that the \emph{convex core} $C(X)$ of a hyperbolic surface $X$ is the smallest closed convex subsurface with boundary which has the same homotopy type as $X$. We denote by $d_X$, or simply $d$ when no confusion can arise, the hyperbolic distance on $X$.

Let $\alpha$ and $\beta$ be two closed geodesics on $X$ and $k\geq 2$ an integer. We say that $\beta$ is the \emph{$k$-fold iterate} of $\alpha$ if, for suitable parametrizations $\alpha:S^1\to X$ and $\beta: S^1\to X$, $\beta(t)=\alpha(kt)$, where $S^1=\R/\Z$. A closed geodesic $\alpha$ is \emph{primitive} if there is no $k\geq 2$ such that $\alpha$ is the $k$-fold iterate of another closed geodesic. A closed curve $\alpha$ is \emph{essential} if it is not homotopic to a point or to a simple closed curve bounding a once-punctured disk. If $\alpha$ is an essential closed curve on a hyperbolic surface $X$, it has a unique geodesic representative in its homotopy class. We denote by $\ell_X(\alpha)$, or simply $\ell(\alpha)$ when no confusion can arise, the length of the geodesic representative of $\alpha$ on $X$.

The \emph{length spectrum} of a hyperbolic surface $X$ is the collection $\LL(X)$ of all lengths of primitive closed geodesics, counted with multiplicity. Two hyperbolic surfaces $X$ and $Y$ are \emph{isospectral} if $\LL(X)=\LL(Y)$ (as multisets).

The following is a fundamental result about hyperbolic surfaces (Keen \cite{keen_collars}; see also \cite[Chapter 4]{buser_geometry}):

\begin{thm}[Collar lemma, Keen (\cite{keen_collars})]\label{lem:collar}
Let $X$ be a hyperbolic surface and $\alpha$ a simple closed geodesic. Then the neighborhood of $\alpha$ given by
$$\mathcal{C}(\alpha)=\{p\in X\st d(p,\alpha)<w(\alpha)\},$$
where
$$w(\alpha)=\arcsinh\left(\frac{1}{\sinh(\ell(\alpha)/2)}\right),$$
is an embedded cylinder.
\end{thm}

We will repeatedly use, without explicit mention, the following consequence of the collar lemma:
\begin{cor}
Let $X$ be a hyperbolic surface. Two distinct simple closed geodesics of length less than $2\arcsinh(1)$ are disjoint.
\end{cor}

A \emph{pants decomposition} of a surface $S$ is a collection of simple closed curves $\mathcal{P}$ such that $S\ssm \mathcal{P}$ is a union of \emph{pairs of pants}, that is, surfaces homeomorphic to a sphere with three boundary components and/or punctures. Note that if a surface has boundary components, we won't consider them as part of the pants decomposition. Given a pants decomposition $\mathcal{P}$ of a surface $S$, to specify a hyperbolic structure on $S$ it is enough to assign to every $\gamma\in \mathcal{P}$ a length $l_\gamma>0$ and a twist parameter $t_\gamma\in \R$. If the surface has boundary, we also need to assign lengths to its boundary components. These are called \emph{Fenchel--Nielsen parameters} and we refer to \cite[Chapter 3]{buser_geometry} for their precise definition.

Given two hyperbolic surfaces $X$ and $Y$ and $K\geq 1$, a homeomorphism $\theta:X\to Y$ is said to be \emph{$K$-quasiconformal} if it is differentiable almost everywhere and, where the derivatives are defined, we have:
$$|\theta_{\bar{z}}|\leq\frac{K-1}{K+1}|\theta_z|.$$

Two hyperbolic surfaces are \emph{quasiconformal} if there is some $K\geq 1$ and a $K$-quasiconformal homeomorphism between them. We recall the following consequence of $K$-quasiconformality, due to Wolpert \cite{wolpert_length}: 

\begin{thm}[Wolpert (\cite{wolpert_length})]
Let $\theta:X_1\to X_2$ be a $K$-quasiconformal map between hyperbolic surfaces. Then for every simple closed geodesic $\gamma$ in $X_1$ we have:
$$\frac{1}{K}\leq \frac{\ell_{X_2}(\theta(\gamma))}{\ell_{X_1}(\gamma)}\leq K.$$
\end{thm}

\section{Surfaces with a given isometry (sub)group}\label{sec:constructions}
In \cite[Sections 3 and 4.2]{apv_isometry}, Aougab, Patel and Vlamis gave constructions of hyperbolic surfaces with a given isometry group $G$. Roughly speaking, the idea is to construct \emph{vertex} surfaces and \emph{edge} surfaces and glue them together according to the combinatorics given by the Cayley graph of $G$, with $G$ as generating set. For our purposes, it will be enough that the given group is a \emph{subgroup} of the group of (orientation preserving) isometries of the hyperbolic surface. This means that we will follow the same construction, but allow more freedom in the choice of hyperbolic structures on the vertex and edge surfaces.

Given a hyperbolic surface $X$, $\Isom^+(X)$ denotes the group of orientation preserving isometries of $X$.

\subsection{Finite isometry group}\label{sec:finitegroup}
Let $G$ be a countable group and $S$ an infinite-type surface without planar ends. We first give the construction of vertex and edge surfaces and then show how to glue them together to obtain the required hyperbolic surface.
\subsubsection*{Vertex surfaces}
All vertex surfaces will be copies of the same surface $V$, constructed as follows. Identify $\Ends(S)$ with a subset $E$ of the sphere $S^2$ and let $S^2_E$ be the complement of $E$ in $S^2$. Choose a basis for the first homology $H^1(S^2_E;\R)$ given by classes of pairwise disjoint simple closed curves $\{c_i\st i\in \delta\}$, where $\delta\in\omega+1$ is such that $\dim H^1(S^2_E;\R)=|\delta|$. Pick pairwise disjoint tubular neighborhoods $\nu_i$ of $c_i$ and let
$$S':=S^2_E\ssm\bigcup_{i\in\delta}\nu_i. $$
Then $S'$ is a disjoint union of connected surfaces $S_j$, $j\in J$, for some index set $J$, each homeomorphic to the plane with some disks removed. Let $\alpha_j\in\omega+1$ be such that $S_j$ has $|\alpha_j|$ boundary components.

Given $\alpha\in\omega+1$, let $Z_\alpha$ be the surface
$$Z_\alpha=\R^2\ssm \bigcup_{m\in \alpha}B\left((0,m),\frac{1}{4}\right),$$
where $B(x,r)$ is the ball of radius $r$ and center $x$. For every $j\in J$, $S_j$ is homeomorphic to $Z_{\alpha_j}$; we fix a homeomorphism $\varphi_j:Z_{\alpha_j}\to S_j$.

Fix also an injection $f$ of $G$ into $\Z$ and define $Z^G_\alpha$ to be
$$Z_\alpha^G = Z_\alpha \ssm \left( \bigcup_{h\in G} \left( \bigcup_{m \in \N} B\left( (f(h), m), \frac14\right) \right) \right)\subset Z_\alpha.$$

Topologically, $V$ is the surface obtained as

$$V=\left( \bigcup_{j\in J} \varphi_j \left(Z_{\alpha_j}^G\right) \right) \cup \left( \bigcup_{i\in \delta} \nu_i \right).$$

Let $\partial(j,h,m):=\varphi_j\left(\partial B\left( (f(h), m), \frac14\right))\right)$, for $j\in J, h\in G$ and $m\in \N$. Let
$$\lambda:J\times G\times \N\to (0,2\arcsinh(1))$$
be an injective function and assign to $\partial(j,h,m)$ the length $\lambda(j,h,m)$. Set
$$\Lambda:=\lambda(J\times G\times \N)$$
and choose pairwise distinct lengths in $(0,2\arcsinh(1))\setminus \Lambda$ for the curves in $\mathcal{P}$. Choose twist parameters for the curves in $\mathcal{P}$ freely. Set
$$P:=\{\ell(\gamma)\st\gamma\in\mathcal{P}\}.$$
\subsubsection*{Edge surfaces}
Choose an injective function $\mu: \N\to (0, 2\arcsinh(1))\setminus\left(\Lambda\cup P\right)$ and let $M:=\mu(\N)$. For every $(j,h,m)\in J\times G\times \N$, let \( E(j,h,2m) \) be a hyperbolic surface obtained by gluing together two pairs of pants, one with boundary lengths \( \lambda(j,h,2m), \mu(m), \) and \( \mu(m) \), and the other with lengths \( \lambda(j,h,2m-1), \mu(m), \) and \( \mu(m) \), with freely chosen twist parameters.
\subsubsection*{Gluing the pieces}
For every $g\in G$, let $V_g$ be a copy of $V$. For every $h\in G\ssm\{e\}$, and for every $(j,m)\in J\times \N$, take a copy of \( E(j,h,2m) \) and glue the boundary component of $E(j,h,2m)$ of length \( \lambda(j,h,2m) \) to the boundary component $\partial(j,h,2m)$ of \( V_g \) and the other boundary component to the boundary component $\partial(j,h,2m-1)$ of \( V_{gh} \).  Choose twist parameters so that there is an action of $G$ by isometries, given by rigidly permuting the vertex surfaces according to the requirement
$$h\cdot V_g=V_{gh}.$$
The resulting surface is denoted $X^G_S$. Note that, by the collar lemma, the only primitive closed geodesics of length less than $2\arcsinh(1)$ are the curves in the pants decomposition of the vertex surfaces, the boundary curves of each vertex surface and the two curves of length in $M$ in each edge surface.

We record the following properties of $X^G_S$, which are proven by Aougab, Patel and Vlamis.
\begin{thm}[Aougab--Patel--Vlamis \cite{apv_isometry}]\label{thm:propertiesofX}
Let $G$ be finite. Then $X^G_S$ is a complete hyperbolic surface of the first kind and $G$ is a subgroup of $\Isom^+(X^G_S)$, acting freely on $X^G_S$. Both $X^G_S$ and the quotient manifold $X^G_S/G$ are homeomorphic to $S$ and $G$ acts trivially on $\Ends(X^G_S)$. If $G$ is countably infinite, $X^G_S$ and $X^G_S/G$ are homeomorphic to the Loch Ness monster.
\end{thm}

Note that, while Aougab, Patel and Vlamis don't state explicitly that the $G$-action is free, it is implicit in their work.






\subsection{Countably infinite isometry group}\label{sec:infinitegroup}
Suppose now that $S$ is an infinite-genus surface with self-similar endspace. As mentioned in Section \ref{sec:prelim-topology}, if $|\Ends(S)|\neq 1$, self-similarity is equivalent to saying that there is a star poin $x\in \Ends(S)$ and a collection of pairwise disjoint noncompact subsets $E_n\subset\Ends(S)$ such that
$$\Ends(S)\ssm\{x\}=\sqcup_{n\in \N}E_n$$
and for every $n,m\in \N$
$$(E_n,E_n\cap\Ends_g(S))\simeq (E_n,E_n\cap\Ends_g(S)).$$
We remark that $x$ is a nonplanar end in this case: since $S$ has infinite genus, there is at least one nonplanar end $y$; if $y=x$ we are done, otherwise $y\in E_n$ for some $n$ and thus the orbit of $y$ contains $x$ as a limit point, so $x$ is also nonplanar\footnote{This argument was communicated to the author by Nick Vlamis.}.

Given a countably infinite group $G$ and $S$ a surface of infinite genus with self-similar endspace and different from the Loch Ness monster, Aougab, Patel and Vlamis construct a hyperbolic surface $Y^G_S$ with isometry group $G$ with a similar argument as in the finite group case. The main modification is in the definition of the vertex surfaces. We will follow the same construction, again allowing for more general hyperbolic structures.

\subsubsection*{Vertex surfaces}
The vertex surface $V$ is a surface such that:
\begin{itemize}
\item $(\Ends(V),\Ends_g(V))\simeq (\overline{E}_1,\overline{E}_1\cap\Ends_g(S))$,
\item $V$ has countably many boundary components, indexed by $G\times \N$ (and denoted, similarly to before, $\partial(g,m)$, for $g\in G$ and $m\in \N$), which don't accumulate anywhere in the surface and such that the only end accumulated by boundary components is the point corresponding to $x$.
\end{itemize}
As in the finite group case, pick an injective function $\lambda:G\times\N\to (0,2\arcsinh(1))$ and let $\partial(g,m)$ have length $\lambda(g,m)$. Let $\Lambda:=\lambda(G\times\N)$. Fix a collection $\mathcal{P}$ of curves in $V$ which form a pants decomposition and assign them pairwise distinct lengths in $(0,2\arcsinh(1))\setminus \Lambda$. Fix a hyperbolic structure by choosing the twist parameters freely.
\subsubsection*{Edge surfaces}
The edge surface $E(h,2m)$, for $h\in G$ and $m\in\N$, is obtained by gluing together two pairs of pants, the first with boundary lengths $\lambda(h,2m),\mu(m),\mu(m)$ and the second with boundary lengths $\lambda(h,2m-1), \mu(m),\mu(m)$, where $\mu$ is a fixed injective function $\N\to (0,2\arcsinh(1))\setminus(\Lambda\cup P)$. Let $M:=\mu(\N)$.
\subsubsection*{Gluing the pieces}
For every $g\in G$, let $V_g$ be a copy of $V$, and as in the previous section, for every $g, h\in G$, $h\neq e$, and $m\in \N$, we glue a copy of $E(h,2m)$ to the boundary component $\partial(h,2m)$ of $V_g$ and to the boundary component $\partial(h,2m-1)$ of $V_{gh}$. Let $Y^G_S$ be the resulting surface, where twist parameters are chose so that there is an action of $G$ on $Y^G_S$ by isometries satisfying
$$h\cdot V_g=V_{gh}.$$

We have:
\begin{thm}[Aougab--Patel--Vlamis (\cite{apv_isometry})]
The surface $Y^G_S$ is of the first kind and homeomorphic to $S$. Furthermore, $G$ is a subgroup $\Isom^+(Y^G_S)$ and the action of $G$ is free.
\end{thm}

\section{Almost conjugate subgroups}\label{sec:groups}

Let $G$ be a finite group. Two subgroups $H_1$ and $H_2$ of $G$ are \emph{almost conjugate} if for every $g\in G$:
$$|[g]\cap H_1|=|[g]\cap H_2|,$$
where $[g]$ is the conjugacy class of $g$ in $G$. Conjugate subgroups are almost conjugate, but the converse doesn't hold. Moreover, given two almost conjugate subgroups $H_1$ and $H_2$ of a finite group $G$, we can get large families of almost conjugate subgroups by looking at the direct product of groups $G^m=G\times\dots\times G$ of $m$ copies of $G$: indeed, for every choice of function $$\varphi:\{1,\dots, m\}\to\{1,2\},$$ the subgroups
$H_\varphi=\prod_{i=1}^m H_{\varphi(i)}$ are pairwise almost conjugate (see \cite[Section 12.6]{buser_geometry}).

We will use a specific example of finite group with almost conjugate subgroups, described in \cite[Example 11.2.2]{buser_geometry}.
\begin{ex}\label{ex:almostconjugate}
Let $G=(\Z/8\Z)^*\ltimes \Z/8\Z$, where $(a,b)\cdot(a',b')=(aa',ab'+b)$, and consider the subgroups
$$H_1=\{e=(1,0),h_1=(3,0),h_2=(5,0),(7,0)\}$$
and
$$H_2=\{(1,0),h_3=(3,4),h_4=(5,4),(7,0)\}.$$
Then $H_1$ and $H_2$ are almost conjugate, but not conjugate. Moreover, since $h_1$ and $h_2$ generate $H_1$, for every $g\in G$, $\{gh_1,gh_2\}\not\subseteq H_2g$. Similarly, for every $g\in G$, $\{gh_3,gh_4\}\not\subseteq H_1g$
\end{ex}
\section{Proof of Theorem \ref{thm:largefamilies}}\label{sec:largefamilies}
This section is dedicated to the proof of Theorem \ref{thm:largefamilies}. Fix an infinite-type surface $S$ without planar ends. We first construct a single family of $2^n$ isospectral, non-isometric hyperbolic structures on $S$. We will discuss how to get an infinite dimensional space of such families at the end of the section.

\subsection{A family of $2^n$ isospectral surfaces}
Let $G,H_1,H_2,h_1$ and $h_2$ be as in Example \ref{ex:almostconjugate}. Let $n\in \N$ and consider $G^n$. For every function $\psi:\{1,\dots, n\}\to\{1,2\}$ consider the subgroup $K_\psi:=\prod_{i=1}^nH_{\psi(i)}$. Denote by $\iota_i:G\to G^n$ the homomorphism identifying $G$ with the $i$-th factor in $G$ (i.e.\ $\iota_i(g)$ is the vector where all entries are $e$, except for the $i$-th, which is $g$).

Let $X=X^{G^n}_S$ be constructed as in Section \ref{sec:finitegroup}, with the following extra condition on $\mathcal{P}$: pick $j_0\in J$ and assume that $\mathcal{P}$ contains curves $q_{1,i},q_{2,i},q_{3,i}$, for every $i\in\{1,\dots,n\}$, where
\begin{itemize}
\item $q_{1,i}$ forms a pair of pants with $\partial(j_0,\iota_i(h_1),2)$ and $\partial(j_0,\iota_i(h_1),1)$,
\item $q_{2,i}$ forms a pair of pants with $\partial(j_0,\iota_i(h_2),2)$ and $\partial(j_0,\iota_i(h_2),1)$,
\item $q_{3,i}$ forms a pair of pants with $q_{1,i}$ and $q_{2,i}$.
\end{itemize}
Set $X_\psi:=X/K_\psi$. Since by Theorem \ref{thm:propertiesofX} $G^n$ acts freely on $X$ and trivially on the space of ends, the $X_\psi$ are all hyperbolic surfaces homeomorphic to $S$.

Note that $X_\psi$ is obtained by gluing vertex surfaces and edge surfaces and that the vertex surfaces correspond to the left cosets of $K_\psi$ in $G^n$. Let $\nu_\psi\subset X_\psi$ be the multicurve given by all geodesics of length in $M$.

As the $K_\psi$ are pairwise almost conjugate, by Theorem \ref{thm:sunada} the $X_\psi$ form a family of isospectral surfaces. So we only have to prove that no two $X_\psi$ are isometric. To this end, consider $X_{\psi_1}$ and $X_{\psi_2}$, for $\psi_1\neq \psi_2$. Let $i\in\{1,\dots, n\}$ be such that $\psi_1(i)\neq \psi_2(i)$. Without loss of generality, assume $\psi_1(i)=1$ and $\psi_2(i)=2$.

Look at the vertex surface in $X_{\psi_1}$ corresponding to the coset $K_{\psi_1}$. By our assumptions, there is a simple closed geodesic $\alpha_i$ which intersects:
\begin{itemize}
\item $\partial(j_0,\iota_i(h_1),2),\partial(j_0,\iota_i(h_1),1),\partial(j_0,\iota_i(h_2),2),\partial(j_0,\iota_i(h_2),1)$, each once,
\item $\mathcal{P}$ only in $q_{1,i}$ and $q_{2,i}$, twice each,
\item $\nu_{\psi_1}$ twice,
\end{itemize}
and such that $\alpha_i\ssm \nu_{\psi_1}$ is contained in a single component of $X_{\psi_1}\ssm \nu_{\psi_1}$. See Figure \ref{fig:curve} for an example of such a curve.

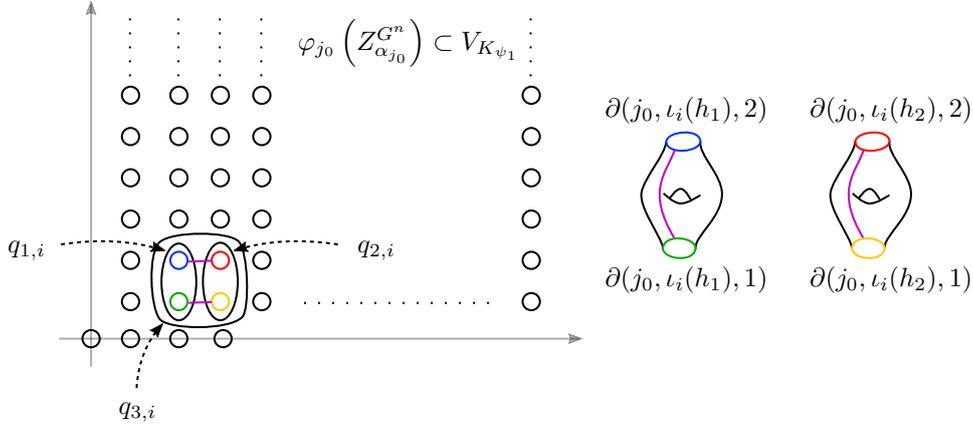
\begin{figure}[h!]
\begin{center}
\begin{overpic}{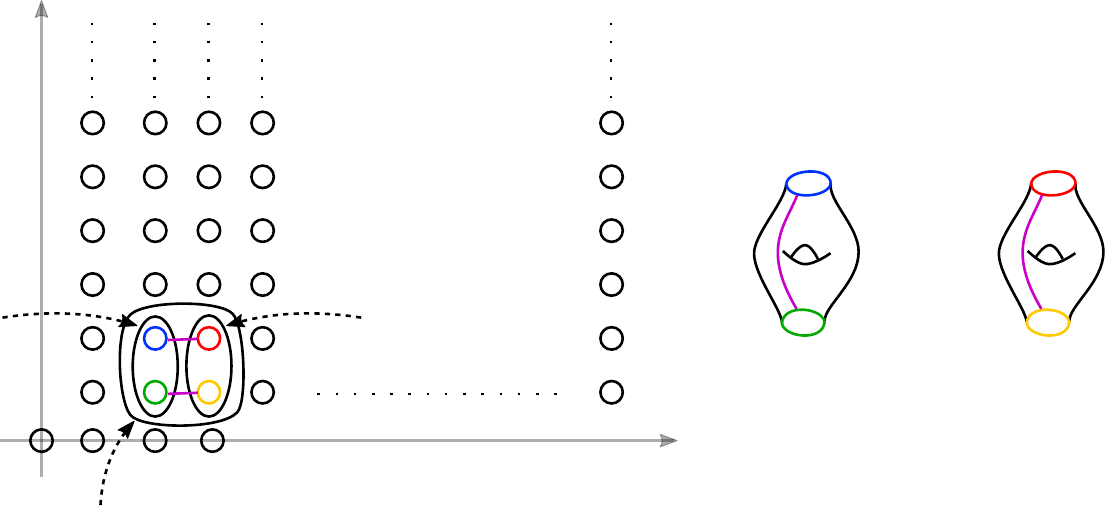}
\put(28,40){$\varphi_{j_0}\left(Z_{\alpha_{j_0}}^{G^n}\right)\subset V_{K_{\psi_1}}$}
\put(64,32){$\partial(j_0,\iota_i(h_1),2)$}
\put(64,12){$\partial(j_0,\iota_i(h_1),1)$}
\put(88,32){$\partial(j_0,\iota_i(h_2),2)$}
\put(88,12){$\partial(j_0,\iota_i(h_2),1)$}
\put(-6,16){$q_{1,i}$}
\put(35,16){$q_{2,i}$}
\put(7,-3){$q_{3,i}$}
\end{overpic}
\vspace{.3cm}
\caption{A schematic picture of the curve $\alpha_i$, in purple, where the surfaces to the right are copies of $E(j_0,\iota_i(h_1),2)$ and $E(j_0,\iota_i(h_2),2)$, glued to $V_{K_{\psi_1}}$ along the curves with matching colors.}\label{fig:curve}
\end{center}
\end{figure}

Suppose by contradiction that there is an isometry $\theta:X_{\psi_1}\to X_{\psi_2}$. Let $\beta=\theta(\alpha_i)$. Note that $\theta(\nu_{\psi_1})=\nu_{\psi_2}$. So, by the condition on $\alpha_i\ssm\nu_{\psi_1}$, $\beta\ssm\nu_{\psi_2}$ is contained in a single component of $X_{\psi_2}\ssm \nu_{\psi_2}$. This means that $\beta$ needs to intersect once each curve 
$$\partial(j_0,\iota_i(h_1),2),\partial(j_0,\iota_i(h_1),1),\partial(j_0,\iota_i(h_2),2),\partial(j_0,\iota_i(h_2),1)$$ in a vertex surface of $X_{\psi_2}$. Let $g=(g_1,\dots,g_n)\in G^n$ be such that the vertex surface intersecting $\beta$ corresponds to the coset $K_{\psi_2}g$. The curves
$$\partial(j_0,\iota_i(h_1),2),\partial(j_0,\iota_i(h_1),1),\partial(j_0,\iota_i(h_2),2),\partial(j_0,\iota_i(h_2),1)$$
in $V_{K_{\psi_2}}$ need to be connected in pairs by edge surfaces $E(j_0,\iota_i(h_1),2)$ and $E(j_0,\iota_i(h_2),2)$. This means that
$$K_{\psi_2}g\iota_i(h_1)=K_{\psi_2}g$$
and
$$K_{\psi_2}g\iota_i(h_2)=K_{\psi_2}g.$$
In particular, looking at the $i$-th component, we need to have
$H_2g_i=H_2g_i h_1$ and $H_2g_i=H_2g_i h_2$, i.e.\ $g_ih_1,g_ih_2\in H_2g_i$, which is impossible.
So the surfaces are pairwise non-isometric.
\subsection{An infinite-dimensional space}\label{sec:infinitedim}
Let $X^{G^n}_S$ be a surface obtained as in the previous section, with the extra condition that the lengths of the curves in $\mathcal{P}$ are isolated in the set (not \emph{multi}set) of lengths of primitive closed geodesics in $X^{G^n}_S$. This can be done since we know which curves have length less than $2\arcsinh(1)$ and we have a lot of freedom in choosing these lengths.

Then we can find pairwise disjoint open intervals $(\gamma)\in I_\gamma\subset (0,2\arcsinh(1))$, for $\gamma\in\mathcal{P}$, such that:
\begin{itemize}
\item $\ell_{X^{G^n}_S}(\gamma)\in I_\gamma$, and
\item if $\delta$ is a primitive closed geodesic with $\ell_{^{G^n}_S}(\delta)\in I_\gamma$, then $\ell_{X^{G^n}_S}(\delta)=\ell_{X^{G^n}_S}(\gamma)$.
\end{itemize}
 
We can vary the metric of $X^{G^n}_S$ by varying simultaneously, for every $\gamma\in\mathcal{P}$, the lengths of all copies of $\gamma$ (one per vertex surface), staying in the interval $I_\gamma$. For any two $\gamma_1,\gamma_2\in\mathcal{P}$ we can vary the lengths independently, so we get an infinite-dimensional family of hyperbolic structures on $X^{G^n}_S$. By taking the quotients by the $K_\psi$ we obtain the desired infinite-dimensional family of surfaces.

\section{Proof of Theorem \ref{thm:selfsimilar}}\label{sec:selfsimilar}
The goal of this section is to prove Theorem \ref{thm:selfsimilar}. Consider $G, H_1, H_2, h_1, h_2, h_3$ and $h_4$ as in Example \ref{ex:almostconjugate}. Let $G^\infty$ be the direct product of countably many copies of $G$ and for every function $\psi: \N\to\{1,2\}$ let $K_\psi$ be the subgroup $K_\psi=\prod_{k=1}^{\infty}H_{\psi(i)}$. Denote by $\iota_i$ the identification of $G$ with the $i$-th factor in $G^\infty$.

If $S$ is the Loch Ness monster (i.e.\ the surface with a single end, which is nonplanar), let $X$ be the surface constructed as in Section \ref{sec:finitegroup}. Note that in this case the vertex surface is homeomorphic to $Z^{G^\infty}_1$ and the boundary components are indexed by $G^\infty\times \N$, so we denote them simply by $\partial(h,m)$, for $h\in G^\infty$ and $m\in \N$. Then it follows from Theorem \ref{thm:propertiesofX} that $X_\psi:=X/K_\psi$ is a Loch Ness monster for every $\psi$.

If $S$ is an infinite-genus surface with self-similar endspace and different from the Loch Ness monster, we let $X$ be the surface constructed from $S$ and $G^\infty$ as in Section \ref{sec:infinitegroup} and $X_\psi:=X/K_{\psi}$. Note that $K_\psi$ has infinite index in $G^\infty$ for every $\psi$. This implies that $X_\psi$ contains infinitely many vertex surfaces, corresponding to the left cosets of $K_\psi$ in $G^\infty$. Then the same argument used by Aougab, Patel and Vlamis to prove that $X$ is homeomorphic to $S$ (\cite[Lemma 4.8]{apv_isometry}) shows that each $X_\psi$ is homeomorphic to $S$ as well.

We now add some extra assumptions to be able to show that the (uncountable) family $$\{X_\psi\st \psi\in\{0,1\}^\N\}$$ satisfies all the conditions in the theorem:
\begin{enumerate}
\item we require that, for every $j\in\{1,2,3,4\}$, $i,m\in\N$, $\mathcal{P}$ contains:
\begin{itemize}
\item a curve $q_{i,j}$ forming a pair of pants with $\partial(\iota_i(h_j),2i)$ and $\partial(\iota_i(h_j),2i-1)$;
\item a curve $r_{i}$ forming a pair of pants with $q_{i,1}$ and $q_{i,2}$ and a curve $s_{i}$ forming a pair of pants with $q_{i,3}$ and $q_{i,4}$;
\end{itemize}
\item we choose $\lambda$, $\mu$ and the lengths of the curves in $\mathcal{P}$ so that $\Lambda\cup P\cup M=\left\{\frac{\arcsinh(1)}{2^{m^2}}\;\middle|\; m\in \N\right\}$.
\end{enumerate}
We call a geodesic \emph{short} if it has length less than $\arcsinh(1)$.

\begin{rmk}
If we only want to construct an uncountable isospectral family, forgetting about getting an infinite-dimensional family and about the quasiconformality statement, we can impose less strict conditions on the curves in $\mathcal{P}$ (similarly to what we do in the proof of Theorem \ref{thm:largefamilies}). Moreover, condition $(2)$ is necessary only to show the quasiconformality statement.
\end{rmk}

The same argument as in the proof of Theorem \ref{thm:largefamilies} shows that the $X_\psi$ are pairwise noisometric.

As mentioned in the introduction, we cannot directly apply Theorem \ref{thm:sunada} to prove isospectrality, but we will use the technique of \emph{transplantation of geodesics} (\cite{buser_isospectral},\cite{berard_transplantation}). We will divide the collection of primitive closed geodesics into (pairwise disjoint) subsets of curves of the same \emph{type} and show that the cardinality of the set of curves of some type in $X_{\psi_1}$ is the same as the cardinality of the set of curves of the same type in $X_{\psi_2}$. Curves of the same type will have by construction the same length, so this will give us a length-preserving bijection between the set of primitive closed geodesics of $X_{\psi_1}$ and the set of primitive closed geodesics of $X_{\psi_2}$, showing isospectrality.

To simplify the notation, we will show that $X_1:=X_{\psi_1}$ and $X_2:=X_{\psi_2}$ are isospectral, where 
$$\psi_1(i)=\left\{\begin{array}{ll}
1 & \mbox{if } i=1\\
2 & \mbox{otherwise}
\end{array}
\right.$$
and
$$\psi_2(i)=\left\{\begin{array}{ll}
1 & \mbox{if } i=2\\
2 & \mbox{otherwise}
\end{array}
\right.$$
and $K_i:=K_{\psi_i}$. The same proof holds for any pair of surfaces.

Denote by $\mathcal{B}$ the multicurve given by the boundary components of the vertex surfaces. Let $c$ be a primitive closed geodesic in $X_1$. We define curves \emph{of type $c$} as follows:
\begin{description}
\item[Case 1] $c$ is a boundary curve of a vertex surface. Then we say that a primitive closed geodesic is \emph{of type $c$} if it is a boundary curve of a vertex surface and has the same length as $c$.
\item[Case 2] $c$ has length in $M$. Then a primitive closed geodesic is \emph{of type $c$} if it has the same length as $c$.
\item[Case 3] $c$ is contained in the interior of an edge surface isomorphic to $E(h,2m)$ and does not have length in $M$. We say that a primitive closed geodesic is \emph{of type $c$} if it is a copy of $c$ in an edge surface isometric to $E(h,2m)$.
\item[Case 4] $c$ is contained in the interior of a vertex surface. Then a primitive closed geodesic is of type $c$ if it is a copy of $c$ in another vertex surface.
\item[Case 5] $c$ is not as in any previous case. Then we can parametrize $c:S^1\to X_1$ and write $S^1$ as a union of closed intervals $I_1=[\theta_1,\theta_2],\dots, I_{2n}=[\theta_{2n},\theta_1]$ with disjoint interiors such that $c(\theta_l)$ belongs to $\mathcal{B}$ and $c(\mbox{int}(I_l))$ is disjoint from $\mathcal{B}$ (see Figure \ref{fig:typec}). Furthermore, we require that $c(I_{2l-1})\subset V_{K_1g_l}$ and $c(I_{2l})\subset E_{2l}$, where $E_{2l}$ is an edge surface isometric to $E(h_l,2m_l)$. We say that a primitive closed geodesic is \emph{of type $c$} if it has a similar decomposition, that is $d$ can be parametrized as $c:S^1\to X_i$, for $i=1,2$, and we can write $S^1$ as a union of closed intervals $I_1'=[\theta_1',\theta_n'],\dots, I_{2n}'=[\theta_{2n}',\theta_1']$ with disjoint interiors such that $d(\theta_l')\in\mathcal{B}$, $c(\mbox{int}(I_l'))\cap\mathcal{B}=\emptyset$, $d(I'_{2l-1})$ is a copy of $c(I_{2l-1})$ contained in a vertex surface and $d(I_{2l}')$ is a copy of $c(I_{2l})$ contained in an edge surface.
\end{description}

\begin{figure}[h]
\begin{center}
\begin{overpic}{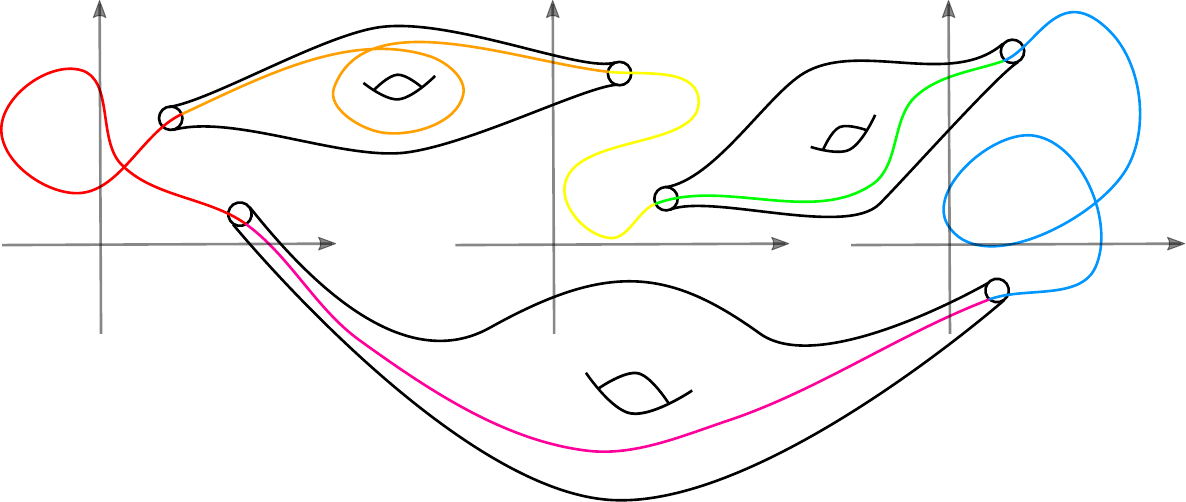}
\put(-7,30){$c(I_1)$}
\put(22,33){$c(I_2)$}
\put(56,37){$c(I_3)$}
\put(63,27){$c(I_4)$}
\put(97,33){$c(I_5)$}
\put(42,8){$c(I_6)$}
\put(1,18){$V_{K_1g_1}$}
\put(39,18){$V_{K_1g_2}$}
\put(72,18){$V_{K_1g_3}$}
\put(29,26){$E(h_1,2m_1)$}
\put(66,38){$E(h_2,2m_2)$}
\put(70,3){$E(h_3,2m_3)$}
\end{overpic}
\caption{A schematic picture of a decomposition of a curve as in case 5}\label{fig:typec}
\end{center}
\end{figure}
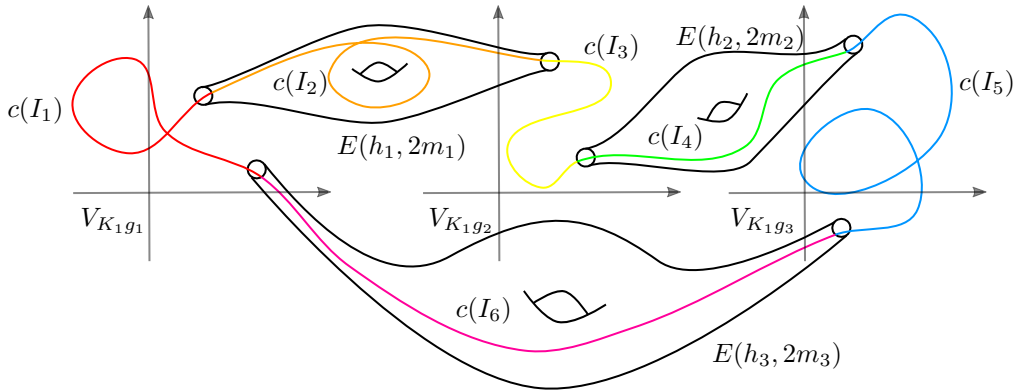

As mentioned before, by construction curves of the same type have the same length. Moreover:
\begin{description}
\item[Case 1] there is exactly one curve of type $c$ in each vertex surface of $X_1$ and in each vertex surface of $X_2$ and both $X_1$ and $X_2$ have countably infinitely many vertex surfaces. So there are countably infinitely many curves of type $c$ in $X_1$ and in $X_2$.
\item[Case 2] a curve has length in $M$ if and only if it is one of the two curves in the pants decomposition of some copy of $E(h,2m)$. Since both $X_1$ and $X_2$ have countably infinitely many edge surface isometric to $E(h,2m)$s, there are countably infinitely many curves of type $c$ in both $X_1$ and $X_2$.
\item[Case 3] curves of type $c$ are in bijection with the edge surfaces isometric to $E(h,2m)$ and again there are countably infinitely many of these in both $X_1$ and $X_2$.
\item[Case 4] curves of type $c$ are in bijection with the vertex surfaces, of which there are countably infinitely many in both $X_1$ and $X_2$.
\end{description}
The last case is the one that requires some more attention (and that relies on the fact that $H_1$ and $H_2$ are almost conjugate).

Assume that $c$ is as in case 5. By construction of $X_1$, since $c$ is a closed curve we need to have
$$K_1g_1h_1\dots h_n=K_1g_1.$$
Moreover, given $c$ and any vertex surface $V_{K_ig}\subset X_i$, for $i=1$ or $2$, we can construct a geodesic obtained by starting with a copy of $c(I_1)$ in the given vertex surface, followed by a copy of $C(I_2)$ in the edge surface glued to the boundary component containing $c(\theta_2)$ (i.e.\ the only possible choice to get a connected path) and so on. We say that such a geodesic is of type $c$. Note that in general such a geodesic is not necessarily closed nor primitive.

Given a group $\Gamma$, a subgroup $\Lambda$ and an element $\gamma\in\Gamma$, we denote by $F(\Lambda,\gamma)$ the collection
$$F(\Lambda,\gamma):=\{\Lambda\delta\st\delta\in\Gamma, \Lambda\delta\gamma=\Lambda\delta\}.$$

A key fact that we will use is the following (see \cite[Proposition 6]{berard_transplantation}):
\begin{lemma}\label{lem:almostconjugate}
Two subgroups $\Lambda_1$ and $\Lambda_2$ of a finite group $\Gamma$ are almost conjugate if and only if for every $\gamma\in\Gamma$
$$|F(\Lambda_1,\gamma)|=|F(\Lambda_2,\gamma)|.$$
\end{lemma}

With this lemma we can prove a more precise statement:
\begin{lemma}\label{lem:almostconjugate+}
Suppose $\Lambda_1$ and $\Lambda_2$ are almost conjugate subgroups of a finite group $\Gamma$. Then for every $n\in \N$ and every $\gamma_1,\dots,\gamma_n\in \Gamma$ there is a bijection $\rho:F(\Lambda_1,\gamma_1\dots\gamma_n)\to F(\Lambda_2,\gamma_1\dots\gamma_n)$ such that the following holds: for every $p$ which divides $n$ and such that $\gamma_i=\gamma_j$ if $i\equiv j$ modulo $\frac{n}{p}$, $\Lambda_1\gamma\in F(\Lambda_1,\gamma_1\dots\gamma_{n/p})$ if and only if $\rho(\Lambda_2\gamma)\in F(\Lambda_2,\gamma_1\dots\gamma_{n/p})$.
\end{lemma}

\begin{proof}
Note first that if $p$ satisfies the conditions in the statement, $F(\Lambda_i,\gamma_1\dots\gamma_{n/p})\subset F(\Lambda_i,\gamma_1\dots\gamma)n)$.

Let $p_1<\dots<p_s$ be the integers dividing $n$ such that $\gamma_i=\gamma_j$ if $i\equiv j$ modulo $\frac{n}{p_t}$, for every $t$. By Lemma \ref{lem:almostconjugate}, $|F(\Lambda_1,\gamma_1\dots\gamma_{n/p_s})|=|F(\Lambda_2,\gamma_1\dots\gamma_{n/p_s})|$, so we can choose a bijection between the two sets. Next look at $F(\Lambda_1,\gamma_1\dots\gamma_{n/p_{s-1}})$ and $F(\Lambda_2,\gamma_1\dots\gamma_{n/p_{s-1}})$: they are also in bijection and so are $$F(\Lambda_1,\gamma_1\dots\gamma_{n/p_{s-1}})\cap F(\Lambda_1,\gamma_1\dots\gamma_{n/p_{s}})=F(\Lambda_1,\gamma_1\dots\gamma_{n/q_s}),$$
where $q_s$ is the greatest common divisor of $p_s$ and $p_{s-1}$, and $$F(\Lambda_2,\gamma_1\dots\gamma_{n/p_{s-1}})\cap F(\Lambda_2,\gamma_1\dots\gamma_{n/p_{s}})=F(\Lambda_2,\gamma_1\dots\gamma{n/q_s}).$$
So we can extend the bijection between $F(\Lambda_1,\gamma_1\dots\gamma_{n/p_s})$ and $F(\Lambda_2,\gamma_1\dots\gamma_{n/p_s})$ to a bijection between $F(\Lambda_1,\gamma_1\dots\gamma_{n/p_s})\cup F(\Lambda_1,\gamma_1\dots\gamma_{n/p_{s-1}})$ and $F(\Lambda_2,\gamma_1\dots\gamma_{n/p_s})\cup F(\Lambda_2,\gamma_1\dots\gamma_{n/p_{s-1}})$. We can repeat the argument to get the desired bijection $F(\Lambda_1,\gamma_1\dots\gamma_n)\to F(\Lambda_2,\gamma_1\dots\gamma_n)$.
\end{proof}

Next we look at $G^\infty$:
\begin{lemma}\label{lem:closed&primitive}
For any $g_1,\dots, g_n\in G^\infty$, there is a bijection $\rho:F(K_1,g_1\dots g_n)\to F(K_2,g_1\dots g_n)$ such that the following holds: for every $p$ which divides $n$ and such that $g_i=g_j$ if $i\equiv j$ modulo $\frac{n}{p}$, $K_1\gamma\in F(K_1,g_1\dots g_{n/p})$ if and only if $\rho(K_1\gamma)\in F(K_2,g_1\dots g_{n/p})$.
\end{lemma}

\begin{proof}
Let $$\rho_1:F(H_1,(g_1\dots g_N)_1)\to F(H_2,(g_1\dots g_N)_1)$$ and $$\rho_2:F(H_2,(g_1\dots g_N)_2)\to F(H_1,(g_1\dots g_N)_2)$$ be bijections as in Lemma \ref{lem:almostconjugate+}. Then it is not difficult to show that
\begin{align*}
\rho:F(K_1,g_1\dots g_n)&\to F(K_2,g_1\dots g_n)\\
(H_1r_1, H_2 r_2, H_2r_3,\dots)&\mapsto (\rho_1(H_1r_1), \rho_2(H_2r_2),H_2r_3,\dots)
\end{align*}
is a bijection satisfying the requirements in the statement.
\end{proof}

Now, if $d$ is a geodesic of type $c$ starting from a vertex surface $V_{K_ig}$, then
\begin{itemize}
\item $d$ is closed if and only if $K_ig\in F(K_i,h_1\dots h_n)$,
\item assuming it is closed, $d$ is the $p$-fold iterate of a closed geodesic if and only if $p$ divides $n$, $c_{2i}=c_{2j}$ and $c_{2i-1}=c_{2j-1}$ if $i\equiv j$ modulo $n/p$, $h_i=h_j$ if $i\equiv j$ modulo $n/p$ and $K_ig\in F(K_i,g_1\dots g_{n/p})$.
\end{itemize}
Furthermore, multiple starting vertex surfaces yield the same curve if and only if there is some $q$ dividing $n$ such that:
\begin{itemize}
\item $c_{2i}$ is a copy of $c_{2j}$ if $i\equiv j$ modulo $q$,
\item $c_{2i-1}$ is a copy of $c_{2j-1}$ if $i\equiv j$ modulo $q$.
\end{itemize}
In particular the number of different vertex surfaces yielding the same primitive closed geodesic $d$ of type $c$ does not depend on $d$ -- denote this multiplicity by $N$. Let $\mathcal{N}$ be the collection of all $p$ dividing $n$ such that $c_{2i}=c_{2j}$ and $c_{2i-1}=c_{2j-1}$ if $i\equiv j$ modulo $n/p$, $h_i=h_j$ if $i\equiv j$ modulo $n/p$ and $K_ig\in F(K_i,g_1\dots g_{n/p})$.

We have shown:
\begin{lemma}
There is an $N$-to-$1$ map
$$F(K_i,g_1\dots g_n)\ssm \bigcup_{p\in\mathcal{N}}F(K_i,g_1\dots g_{n/p})\to \{\mbox{primitive closed geodesics of type $c$ in $X_i$}\}.$$
\end{lemma}

Since $F(K_1,g_1\dots g_n)\ssm \bigcup_{p\in\mathcal{N}}F(K_1,g_1\dots g_{n/p})$ and $F(K_2,g_1\dots g_n)\ssm \bigcup_{p\in\mathcal{N}}F(K_2,g_1\dots g_{n/p})$ are in bijection by Lemma \ref{lem:closed&primitive}, we have a bijection between the set of primitive closed geodesics of type $c$ on $X_1$ and the set of primitive closed geodesics of type $c$ on $X_2$. This concludes the proof of isospectrality of $X_1$ and $X_2$.

Next we want to show that there is an uncountable family in $\left\{X_\psi\;\middle| \;\psi\in\{1,2\}^\N\right\}$ of pairwise not quasiconformal hyperbolic surfaces. We will use the following lemma.

\begin{lemma}
Let $\mathcal{F}$ be a maximal (with respect to inclusion) collection of functions $\psi:\N\to\{1,2\}$ such that for every $\psi_1,\psi_2\in\mathcal{F}$, $\psi_1\neq \psi_2$, there is an infinite sequence $\{i_k\}_{k\in\N}\subset\N$ such that $\psi_1(i_k)\neq \psi_2(i_k)$ for every $k$. Then $\mathcal{F}$ is uncountable.
\end{lemma}
\begin{proof}
It is easy to show that $\mathcal{F}$ is infinite. Suppose by contradiction that $\mathcal{F}$ is countable and denote its elements by $\psi_n$, $n\in \N$. Let $\psi:\N\to\{1,2\}$ defined by $$\psi(2^{n-1}(2i-1))\neq \psi_n(2^{n-1}(2i-1))\;\forall n,i\in\N.$$
So for every $n\in \N$, $\psi_n$ and $\psi$ have infinitely many distinct values, contradicting the maximality of $\mathcal{F}$.
\end{proof}

Using this we prove that, for $\mathcal{F}$ a family as in the lemma, the surfaces $\{X_\psi\st\psi\in\mathcal{F}\}$ are pairwise not quasiconformal to each other.

Given $\psi:\N\to\{1,2\}$, we construct the sequence of curves $\gamma_{\psi,i}$ on $X_\psi$ defined as follows:
\begin{itemize}
\item if $\psi(i)=1$, $\gamma_{\psi,i}$ is a simple closed curve intersecting
\begin{itemize}
\item $\partial(\iota_i(h_1),2i),\partial(\iota_i(h_1),2i-1),\partial(\iota_i(h_2),2i),\partial(\iota_i(h_2),2i-1)$, each once,
\item $\mathcal{P}$ only in $q_{1,i}$ and $q_{2,i}$, twice each,
\item $\nu_{\psi}$ twice,
\end{itemize}
and such that $\gamma_{\psi,i}\ssm \nu_{\psi}$ is contained in a single component of $X_\psi\ssm \nu_{\psi}$;
\item if $\psi(i)=2$, $\gamma_{\psi,i}$ is a simple closed curve intersecting
\begin{itemize}
\item $\partial(\iota_i(h_3),2i),\partial(\iota_i(h_3),2i-1),\partial(\iota_i(h_4),2i),\partial(\iota_i(h_4),2i-1)$, each once,
\item $\mathcal{P}$ only in $q_{3,i}$ and $q_{4,i}$, twice each,
\item $\nu_{\psi}$ twice,
\end{itemize}
and such that $\gamma_{\psi,i}\ssm \nu_{\psi}$ is contained in a single component of $X_\psi\ssm \nu_{\psi}$;
\end{itemize}

Suppose by contradiction that there is a $K$-quasiconformal map $\theta:X_{\psi_1}\to X_{\psi_2}$ for $\psi_1\neq\psi_2\in \mathcal{F}$.

The same argument as in the proof of Theorem \ref{thm:largefamilies} shows that the set of lengths of short curves intersected by $\theta(\gamma_{\psi_1,i_k})$ cannot be the same as the set of lengths of short curves intersected by $\gamma_{\psi_1,i_k}$. So for every $k$ there is a short curve $\delta_k$ intersecting $\gamma_{\psi_1,i_k}$ such that $\ell_{X_1}(\delta_k)\neq \ell_{X_2}(\theta(\delta_k))$. Note that by construction $\ell_{X_1}(\delta_k)\to 0$, as $k\to\infty$.

\begin{description}
\item[Case 1] for infinitely many $k$, the geodesic representative of $\theta(\delta_k)$ is long. Since there is some $k$ such that $\ell_{X_1}(\delta_k)=\frac{\arcsinh(1)}{2^{n^2}}$ for some $n>\sqrt{\log_2(K)}$, we have
$$\frac{\ell_{X_2}(\theta(\delta_k))}{\ell_{X_1}(\delta_k)}\geq 2^{2n}>K,$$
a contradiction.
\item[Case 2] for all but finitely many $k$, the geodesic representative of $\theta(\delta_k)$ is short. Since there is some $k$ such that  $\ell_{X_1}(\delta_k)=\frac{\arcsinh(1)}{2^{n^2}}$ for some $n>\frac{\log_2(K)+1}{2}$ and $\ell_{X_2}(\theta(\delta_k))=\frac{\arcsinh(1)}{2^{m^2}}$, for some $m\neq n$, we have
that either $m<n$ and $$\frac{\ell_{X_2}(\theta(\delta_k))}{\ell_{X_1}(\delta_k)}>K$$
or
$m>n$ and
$$\frac{\ell_{X_2}(\theta(\delta_k))}{\ell_{X_1}(\delta_k)}<\frac{1}{K},$$
a contradiction.
\end{description}

So $X_{\psi_1}$ and $X_{\psi_2}$ are not quasiconformal.

To conclude the proof of Theorem \ref{thm:selfsimilar} we just need to show that the surfaces we constructed belong to infinite-dimensional families as required. Since the lengths in $P$ are isolated in the set of lengths of short curves on $X$, we can vary them all slightly to give the required infinite-dimensional family (as in Section \ref{sec:infinitedim}, with extra care to choose small enough intervals so that the proof of not quasiconformality goes through).

\bibliographystyle{alpha}
\bibliography{references}

\end{document}